\documentclass[reqno]{amsart}
\usepackage{a4}
\usepackage{amssymb}
\usepackage{verbatim}
\usepackage{mathrsfs}
\usepackage{lcy}

\newtheorem{Th}{Theorem}[section]
\newtheorem{Cor}[Th]{Corollary}
\newtheorem{Lem}[Th]{Lemma}
\newtheorem{Prop}[Th]{Proposition}

\newtheorem{Rem}[Th]{Remark}

\newtheorem{claim-num}{Claim}

\def\aut#1{\mbox{\rm Aut}(#1)}
\def\av#1{\overline{#1}}
\def\a{\alpha}

\def\cA{\mathcal A}
\def\cB{\mathcal B}

\def\cU{\mathcal U}
\def\cX{\mathcal X}
\def\cY{\mathcal Y}
\def\cZ{\mathcal Z}

\def\cP{\mathcal P}

\def\eps{\varepsilon}

\def\N{\mathbf N}
\def\nc{\mbox{\rm NC}}

\def\s{\sigma}

\def\str#1{\langle#1\rangle}
\def\stuple#1{#1_1,\ldots,#1_s}
\def\sym#1{\mbox{\rm Sym}(#1)}
\def\inn#1{\mbox{\rm Inn}(#1)}

\def\inv{{}^{-1}}
\def\sle{\subseteq}
\def\Mod#1{\ (\mbox{\rm mod}\ #1)}
\def\Z{\mathbf Z}
\def\id{\mbox{\rm id}}

\def\text#1{\mbox{\rm #1}}

\renewcommand{\le}{\leqslant}
\renewcommand{\ge}{\geqslant}

\def\To{\Rightarrow}

\def\vk{\varkappa}

\def\rank{\operatorname{rank} }

\def\ptl{\partial}

\def\avst#1{ \overline{\mathstrut #1} }

\def\avR{\widehat R}

\numberwithin{equation}{section}

\begin{document}

\title[Small conjugacy classes]
{Small conjugacy classes in the automorphism groups of relatively free groups}
\author{Vladimir Tolstykh}
\address{Vladimir Tolstykh\\ Department of Mathematics\\ Yeditepe University\\
34755 Kay\i\c sda\u g\i \\
Istanbul\\
Turkey}
\email{vtolstykh@yeditepe.edu.tr}
\maketitle

\section{Small conjugacy classes}

Recall that a group $G$ is said to be {\it complete}
if all automorphisms of $G$ are inner and the center
of $G$ is trivial: $\aut G=\inn G \cong G.$ In a series of papers \cite{DFo0,DFo2,DFo} Dyer and Formanek justified
several conjectures by Baumslag on the automorphism towers
of finitely generated relatively free groups. In particular,
they proved that the automorphism group of a
free group $F_n$ of finite rank $n \ge 2$ is complete
(and hence the automorphism tower of $F_n$ terminates
already after two steps) \cite{DFo0}
and that the automorphism group $\aut{F_n/R'}$ of the group $F_n/R'$
is complete where $R$ is a characteristic subgroup of $F_n$ which is contained
in the commutator subgroup $F_n'$ of $F_n$ such that the
quotient group $F_n/R$ is residually torsion-free nilpotent
\cite{DFo}.

The aim of the present paper is to extend the latter result to
infinitely generated relatively free groups.

Let $G$ be an infinitely generated relatively
free group, let $\vk$ denote the rank
of $G,$ and let $\Gamma$ denote $\aut G.$
If $\cX$ is a basis of $G,$ we shall denote by $\sym \cX$
the subgroup of automorphisms of $G$
fixing $\cX$ setwise; members
of $\sym \cX$ will be called {\it permutational}
automorphisms of $G$ with respect to $\cX.$

Let $\{A_i : i \in I\}$ be a family
of subgroups of $G.$ We write $G = \circledast_{i \in I} A_i,$
if there is a basis $\cX$ of $G$
such that  $A_i=\str{\cA_i}$ where $\cA_i = \cX \cap A_i$ for every
$i \in I,$ the sets $\cA_i$ are pairwise disjoint,
and their union $\bigcup_{i \in I} \cA_i$
is $\cX.$

Clearly, if $G = \circledast_{i \in I} A_i,$ then given
any family $\{\s_i \in \aut{A_i} : i \in I\}$ of automorphisms
of subgroups $A_i,$ there is a uniquely determined
automorphism $\s$ of $G$ such that $\s|_{A_i} =\s_i$
for all $i \in I;$ we shall denote $\s$ by
$\circledast_{i \in I} \s_i.$

We shall use the standard notation of the theory of infinite
permutation groups. Thus, having a group $H$
acting on a set $X,$ we shall denote by $H_{(Y)}$
and by $H_{\{Y\}}$ the pointwise and the setwise
stabilizer of a subset $Y$ of $X$ in $H,$ respectively. Any symbol
of the form $H_{*_1,*_2}$ denotes the intersection
of subgroups $H_{*_1}$ and $H_{*_2}.$

In this section we shall study elements
of $\aut G$ having `small' conjugacy classes.
We shall say that the conjugacy
class $\s^\Gamma$ of a $\s \in \aut G$ is {\it small}
if the cardinality of $\s^\Gamma$ is at most $\vk=\rank G;$ equivalently,
the index $|\Gamma:C(\s)|=|\s^\Gamma|$ of the centralizer $C(\s)$ in $\Gamma$
is at most $\vk.$

We start with a statement on existence of
certain stabilizers in the subgroups of $\aut G$
having index $\le \rank G.$ Recall that a {\it moiety}
of an infinite set $I$ is any subset $J$ of $I$ with
$|J|=|I\setminus J|.$ The proof of the statement
uses the ideas and methods developed by Dixon, Neumann
and Thomas in \cite{DiNeuTho}.

\begin{Prop} \label{MStab-in-a-small-index-subgroup}
Let $G$ be a relatively free group
of infinite rank $\vk,$ $\cX$ a basis of $G$
and $\Sigma$ a subgroup of the automorphism group $\Gamma=\aut G$
of index at most $\rank G.$ Then

{\rm (i)} there is a subset $\cU$ of $\cX$ of
cardinality less than $\vk$ such that $\Sigma$
contains the subgroup $\sym \cX_{(\cU)}$
of permutational automorphisms with regard
to $\cX$ which fix $\cU$ pointwise;

{\rm (ii)} for every moiety $\cZ$ of $\cX \setminus \cU,$
$\Sigma$ contains the subgroup
$\Gamma_{(\cX \setminus \cZ),\{ \str{\cZ} \}}$
of automorphisms of $G$ which fix $\cX \setminus \cZ$
pointwise and preserve the subgroup $\str{\cZ}$
generated by $\cZ;$

{\rm (iii)} for every moiety $\cZ=\{z_i : i \in I\}$ of
$\cX \setminus \cU,$ and for every subset $\{v_i : i \in I\}$
of the subgroup $\str \cU$ generated by $\cU,$ $\Sigma$
contains an automorphism $\s$ of $G$ which fixes the
set $\cX \setminus \cZ$ pointwise and takes
an element $z_i$ of $\cZ$ to $z_i v_i$:
\begin{alignat}3 \label{U-transvec}
&\s x   &&=x,              &&x \in \cX \setminus \cZ,\\
&\s z_i &&= z_i v_i, \qquad &&i \in I. \nonumber
\end{alignat}
\end{Prop}

\begin{proof} (0). Before proving (i), we are
going first to do some preliminary work for (ii) and (iii).

(a). Let $\cY$ be a subset of $\cX$
of cardinality $\vk=\rank G.$ We partition $\cY$
into $\vk$ moieties:
$$
\cY = \bigsqcup_{i \in I} \cY_i
$$
and then partition every $\cY_i$ into countably
many moities:
$$
\cY_i=\bigsqcup_{k \in \N} \cY_{i,k}.
$$
Fix an index $i_0$ of $I$ and consider any
automorphism $\a$ of the subgroup $\str{\cY_{i_0,0}}$
generated by $\cY_{i_0,0}.$ We extend $\a$ on $\str{\cY_{i_0}}$ as follows:
\begin{equation}
\beta(\alpha) = (\a \circledast \a\inv \circledast \id) \circledast (\a \circledast \a\inv \circledast \id) \circledast \ldots
\end{equation}
Here, as it is quite commonly done for simplicity's
sake, any $\a$ (resp. $\a\inv,$ resp. $\id$)
in the right-hand side of (\theequation) rather {\it indicates} that the action
of the restriction of $\beta(\alpha)$ on a subgroup $\str{\cY_{i_0,k}}$
is isomorphic to the action of $\a$ (resp. $\a\inv,$ resp. $\id$)
on $\str{\cY_{i_0,0}};$ the corresponding isomorphism
of actions is supposed to be induced by a bijection
from $\cY_{i_0,0}$ onto $\cY_{i_0,k}$. Any
$\circledast$-product of automorphisms below
must be treated in a similar fashion.

We then consider the family $\Lambda_\alpha$ of automorphisms of $G$ such
that
\begin{equation}
\lambda=\id_{\str{\cX \setminus \cY}} \circledast (\circledast_{i \in I} \beta(\alpha)^{\varepsilon_i})
\end{equation}
where $\varepsilon_i=0,1$ ($i\in I$).

Observe that $|\Lambda_\alpha|=2^\vk.$
Therefore for any subgroup $H$ of $\Gamma$ having
index less than $2^\vk,$ there are distinct $\lambda_1,\lambda_2 \in \Lambda_\alpha$
with $\lambda_1 \lambda_2\inv \in H.$ Clearly,
the product $\lambda_1 \lambda_2\inv$ is also
of the form (\theequation) with $\eps_i$
equal either to $0,$ or to $1,$ or to $-1$
($i \in I$).

(b). Observe, for future use, that there are permutational
automorphisms $\rho_1,\rho_2 \in \aut{\str{\cY_{i_0}} }$
with regard to the basis $\cY_{i_0}$ such that $\beta(\alpha)^{\rho_1} \beta(\alpha)$ equals the
trivial automorphism of $\str{\cY_{i_0}}$ and
$\beta(\alpha)^{\rho_2} \beta(\alpha)$ equals the automorphism
\begin{equation}
(\a \circledast \id \circledast \id) \circledast  (\id \circledast \id \circledast \id) \circledast \ldots,
\end{equation}
that is, coincides with $\a$ on $\str{\cY_{i_0,0}}.$
The similar arguments applies to $\beta(\a)\inv$:
there are permutational automorphism $\mu_1,\mu_2 \in \aut{\str{\cY_{i_0}} }$
such that $[\beta(\a)\inv]^{\mu_1} \beta(\a)\inv$
is the trivial automorphism of $\str{\cY_{i_0}},$ and
$[\beta(\a)\inv]^{\mu_2} \beta(\a)\inv$
is the automorphism (\theequation).

(i). As
$$
|\sym\cX : \sym \cX \cap \Sigma| \le |\Gamma : \Sigma| \le \rank G=|\cX|,
$$
we get that $\sym \cX \cap \Sigma$ is a subgroup of
$\Pi=\sym \cX$ of index at
most $|\cX|.$ Then by Theorem $2^{\mbox{\normalsize$\flat$}}$
of \cite{DiNeuTho}, there is a subset $\cU$ of $\cX$ of cardinality
less than $|\cX|$ such that
$$
\Pi_{(\cU)} \le \Pi \cap \Sigma=\sym \cX\cap \Sigma.
$$

(ii). We apply the considerations in (0) to the set $\cY =\cX
\setminus \cU.$ We partition $\cY$ as in (0), assuming
that $\cY_{i_0,0}=\cZ.$ Take a $\s \in \Gamma_{(\cX \setminus \cZ),\{\str{\cZ}\}}$
and let
$$
\alpha=\sigma|_{\str{\cZ}}=\sigma|_{\str{\cY_{i_0,0}}}.
$$

By the conclusion made in (a) in (0), as the index
of $\Sigma$ in $\Gamma$ is at most $\vk < 2^\vk,$ $\Sigma$ contains an element
of $\Lambda_\a,$ that is, there are $\varepsilon_i=0,1,-1$
($i \in I$) some of which are nonzero such that
$$
\gamma=\id_{\str{\cU}} \circledast (\circledast_{i \in I} \beta(\alpha)^{\varepsilon_i}) \in \Sigma.
$$
Using (b) in (0), it is then rather easy to find a permutational
automorphism $\pi$ in $\Pi_{(\cU)} =\sym \cX_{(\cU)}$
which fixes $\cU$ pointwise and
such that the action of $\gamma^\pi \gamma$
on some $\cY_{j,0}$
is isomorphic to the action of $\a$
and the action of $\gamma^\pi \gamma$
 on all other sets $\cY_{i,k}$ is trivial.
Indeed, let
$$
I_0=\{i \in I : \eps_i =0\}, \quad
I_1=\{i \in I : \eps_i =1\}, \quad
I_{-1}=\{i \in I : \eps_i =-1\}.
$$
If $I_1$ is empty, we switch to $\gamma\inv,$
also a member of $\Sigma.$ Thus, without
loss of generality we may assume that
$I_1 \ne \varnothing.$ Let us consider the more difficult
case, when $i_0 \notin I_1.$ Take a $j$ in $I_1,$
and define $\pi \in \Pi_{(\cU)}$ as follows:
\begin{equation}
\pi =\id_{(\cU)} \circledast (\circledast_{i \in I_{0}} \id)
\circledast \rho_1
\circledast (\circledast_{i \in I_1, i \ne j} \rho_2)
\circledast (\circledast_{i \in I_{-1}} \mu_2),
\end{equation}
where $\rho_1,\rho_2,\mu_2$ were defined
in (b) in (0), and $\rho_1$ in (\theequation) describes
the action of $\pi$ on $\str{\cY_j}.$

At the next step we conjugate
$\gamma^\pi \gamma \in \Sigma$ by an appropriately chosen
permutational automorphism $\pi_1 \in \Pi_{(\cU)} \le \Sigma,$
interchanging $\cY_{i_0}$ and $\cY_j,$
while fixing pointwise $\cU$ and all other
$\cY_i,$ thereby obtaining $\s.$

(iii). The proof is basically the same as in (i).
Having an automorphism $\s$ of $G$ with \eqref{U-transvec},
we find in $\Sigma$ an automorphism $\s^*$ of $G$ whose
action on each of some $\vk$ moieties of $\cX \setminus \cU,$
into which $\cX \setminus \cU$ is partitioned and
one of which contains $\cZ,$ is either trivial,
or isomorphic to the action of $\beta(\sigma|_{\str{\cZ}})$ (suitably
defined) on the moiety containing $\cZ,$ or
to isomorphic the action of $\beta(\sigma|_{\str{\cZ}})\inv.$
Then for a suitable $\pi,\pi_1 \in \Pi_{(\cU)}$ we have
that an element $((\sigma^*)^\pi \s^*)^{\pi_1}$
of $\Sigma$ is equal to $\s.$
\end{proof}

\begin{Prop} \label{SmNormCls}
Let $G$ be a relatively
free group of infinite rank $\vk,$ $\cX$ a basis of $G.$
The conjugacy class of a $\s \in \aut G$
is small if and only if
there are finitely many elements $u_1,\ldots,u_s$
of $\cX$ and a term $w(*;*_1,\ldots,*_s)$ of the
language of group theory {\rm(}a group
word in symbols $*,\stuple *${\rm)} such that
\begin{equation} \label{w_sets_the_image}
\s(x)=w(x;u_1,\ldots,u_s)
\end{equation}
for all $x \in \cX$ and
\begin{equation} \label{w_is_hom}
w(xy;u_1,\ldots,u_s) = w(x;u_1,\ldots,u_s) \cdot w(y;u_1,\ldots,u_s)
\end{equation}
for all $x,y \in \cX$ {\rm(}in effect, $\s(g)=w(g;u_1,\ldots,u_s)$ for all $g \in G${\rm)}.
\end{Prop}

\begin{proof}
Since the conjugacy class of $\s$ is of cardinality $\le \vk,$
the index of the centralizer $C(\s)$ of $\s$
in $\aut G$ is at most $\vk.$ Hence by Proposition
\ref{MStab-in-a-small-index-subgroup}, for $\Sigma=C(\s)$ there is a subset
$\cU$ of $\cX$ having the properties (i-iii) listed
in this proposition. In particular, $\Pi_{(\cU)} =\sym \cX_{(\cU)} \le C(\s).$

Consider an $x \in \cX \setminus \cU.$ Let $w_x(x; \vec y, \vec u)$
be a term of the language of group theory such that
$$
\s x = w_x(x; \vec y, \vec u_x)
$$
where $\vec y$ is a tuple of elements of $\cX \setminus \cU$
none of which equals $x$ and $\vec u_x$ is a tuple
of elements of $\cU.$

(a) A permutational automorphism
in $\Pi_{(\cU)}$ which fixes $x$ and takes $\vec y$
to a tuple $\pi \vec y$ with $\pi \vec y \cap \vec y =\varnothing$
must commute with $\s.$ Then $\pi \s \pi\inv x = \s x$
implies that
$$
w_x(x; \pi \vec y, \vec u_x) = w_x(x; \vec y, \vec u_x).
$$
Take an endomorphism $\eps$ of $G$ which sends all members
of $\pi \vec y$ to $1$ and fixes all other elements
of $\cX.$ Apply $\eps$ to the both parts of the
last equation:
\begin{align*}
\eps(w_x(x; \pi \vec y, \vec u_x)) =
w_x(\eps(x); \eps(\pi \vec y), \eps(\vec u_x))
=w_x(x;1,\ldots,1, \vec u_x)
=w_x(x; \vec y, \vec u_x).
\end{align*}
It follows that $\s x \in \str{x,\vec u_x},$ and we can
assume that
$$
\s x =w_x(x; \vec u_x).
$$

(b) Take another element $y$ in $\cX \setminus \cU.$
Again, $\s$ must commute with a permutational
automorphism $\rho$ in $\Pi_{(\cU)}$ interchanging
$x$ and $y.$ Comparing $\rho \s \rho\inv x$ and $\s x,$
we obtain that
$$
w_y(x; \vec u_y) = w_x(x; \vec u_x)
$$
But then
$$
w_y(y; \vec u_y) = w_x(y; \vec u_x),
$$
after forcing an endomorphism of $G$
fixing $\cU$ pointwise and taking $x$ to $y$
to act on the both parts of the preceding equation.
So the image $\s y$ of $y$ can be obtained
by replacing occurrences $x$ in $w_x(x;\vec u)$
by $y.$ We arrive therefore at the conclusion that
$$
\s z =w(z; \vec u)
$$
where $w(*;*_1,\ldots,*_s)$ is a fixed term and $\vec u$ is a fixed
tuple of elements of $\cU$ for all $z \in \cX \setminus \cU.$

(c) Let $x,y$ be distinct elements of $\cX \setminus \cU.$
The `transvection' $U$ which takes $x$ to $xy$ and fixes
all other elements of $\cX$ belongs to $C(\s)$ by
part (ii) of Proposition \ref{MStab-in-a-small-index-subgroup}.
The equality $U \s U\inv x = \s x$ implies then that
$$
w(xy; \vec u) w(y; \vec u)\inv = w(x; \vec u),
$$
or
\begin{equation}
w(xy; \vec u)= w(x; \vec u) \cdot w(y; \vec u).
\end{equation}
As $x,y,\vec u$ are all members of some basis of $G,$
\begin{equation} \label{w_is_a_hom_word}
w(ab; \vec u)= w(a; \vec u) \cdot w(b; \vec u).
\end{equation}
for every $a,b \in G$ (after acting on
the both parts of (\theequation) by an endomorphism of $G$
fixing $\vec u$ pointwise and taking
$x$ to $a$ and $y$ to $b$.)

(d). The argument similar to one we have used in (a)
shows that for every $v \in \cU,$ the image
$\s v$ of $v$ is in the subgroup generated
by $\cU.$

Take an $x \in \cX \setminus \cU,$ an element
$v \in \cU$ and another `transvection' $U_1$
which takes $x$ to $xv$ and fixes $\cX \setminus
\{x\}$ pointwise. By part (ii) of
Proposition \ref{MStab-in-a-small-index-subgroup},
$U_1$ commutes with $\s.$ Observe that $U_1(\s v) =\s v,$
since $U_1$ stabilizes all elements of $\cU.$ Hence
$$
w(x;\vec u) =\s x = U_1 \s U_1\inv x =
U_1(\s(x v\inv)) = U_1(\s(x) \s(v\inv))=
w(xv;\vec u) \s(v\inv).
$$
By \eqref{w_is_a_hom_word}, $w(xv;\vec u) =w(x;\vec u) w(v;\vec u),$
whence $\s v = w(v;\vec u),$ completing the proof
of the necessity part.

Conversely, if a term $w(*;*_1,\ldots,*_s)$
and a tuple $\vec u$ of $\cX$ satisfy
\eqref{w_sets_the_image} and \eqref{w_is_hom},
then
$$
\s(g) =w(g;\vec u)
$$
for all $g \in G.$ Let $\pi \in \aut G.$ Hence
$$
\pi \s \pi\inv g=\pi(\s (\pi\inv g))= \pi(w(\pi\inv g,\vec u))=
w(g; \pi \vec u)
$$
for all $g \in G.$ Therefore there are at most $\vk$ conjugates of $\s,$
since there are at most $\vk$ elements in the orbit
of the tuple $\vec u$ under $\aut G.$
\end{proof}

Working with a relatively free group $G$
we shall denote by $\tau_g$
the inner automorphism of $G$ determined
by a $g \in G.$

\begin{Prop} \label{S_is_def}
Let $G$ be a centerless relatively free group
of infinite rank $\vk.$

{\rm (i)} Suppose that the cardinality of the conjugacy
class $\rho^\Gamma$ of a $\rho \in \Gamma$ is less than $\vk.$ Then
$\rho$ is the identity;

{\rm (ii)} the cardinal
$$
\min \{|\pi^\Gamma| : \pi \in \Gamma, \pi \ne \id\}
$$
is equal to $\vk;$

{\rm (iii)} the conjugacy class of a nonidentity $\s \in \Gamma$
is small if and only if
$
|\s^\Gamma| \le |\pi^\Gamma|
$
for every nonidentity $\pi \in \Gamma;$

{\rm (iv)} the subgroup $S$ of all elements
of $\Gamma$ whose conjugacy class is small is
a characteristic subgroup of $\Gamma.$

\end{Prop}

\begin{proof}
(i). Let $\cX$ be a basis of $G.$ By
Proposition \ref{SmNormCls}, there is a term
$w(*;*_1,\ldots,*_s)$ and elements $\vec u=\stuple u \in \cX$
such that
$$
\rho(g)=w(g;\vec u)
$$
for all $g \in G.$ As
$$
|\{\tau_x \rho \tau_x\inv : x \in \cX \setminus \vec u\}| \le
|\rho^\Gamma| < \vk,
$$
there are distinct $x_1,x_2 \in \cX \setminus \vec u$ such that
$$
\tau_{x_1} \rho \tau_{x_1}\inv  =
\tau_{x_2} \rho \tau_{x_2}\inv,
$$
or
$$
\rho \tau_{x_1\inv x_2} \rho\inv = \tau_{x_1\inv x_2}
$$
or $\rho(x_1\inv x_2) = x_1\inv x_2$ because
$G$ is centerless. Therefore
$$
w(x_1\inv x_2; \vec u) = x_1\inv x_2,
$$
whence $w(g;\vec u)=g=\rho(g)$ for all $g \in G,$
since the element $x_1\inv x_2$ and the elements
of the tuple $\vec u$ all occur in a suitable
basis of $G$

(ii). Write $\lambda$ for the cardinal
$$
\min \{|\pi^\Gamma| : \pi \in \Gamma, \pi \ne \id\}
$$
By (i), $\vk \le \lambda.$ On the other hand,
for any inner automorphism $\tau_g$ determined
by a nonidentity element $g \in G$
$$
\lambda \le |\tau_g^\Gamma|=\vk.
$$

(iii). By (ii).

(iv). By (iii).
\end{proof}

\begin{Cor}
Let $\mathfrak V$ be a variety of groups whose
free groups are centerless. Then for any infinitely
generated free groups $G_1,G_2 \in \mathfrak V$
$$
\aut{G_1} \cong \aut{G_2} \iff \rank(G_1) =\rank(G_2).
$$
\end{Cor}

\begin{proof}
By Proposition \ref{S_is_def} (ii).
\end{proof}

In \cite{Fo} Formanek proved the subgroup
$\inn{F_n}$ of the automorphism group $\aut{F_n}$
of a free group $F_n$ of finite
rank $n \ge 2$ is the only free normal
subgroup of $\aut{F_n}$ of rank $n.$
Our next corollary extends this result
to free groups of infinite rank.

\begin{Cor}
Let $F=F_\vk$ be a free group of infinite rank $\vk.$
Then $\s \in \aut{F_\vk}$ has small conjugacy
class if and only if $\s$ is an inner
automorphism of $F_\vk.$ Consequently,
$\inn{F_\vk}$ is the largest {\rm(}free{\rm)} normal subgroup
of $\aut{F_\vk}$ of cardinality $\vk.$
\end{Cor}

\begin{proof}
Let the conjugacy class of a $\s \in \aut F$
be small. Take a basis $\cX$ of $F$ and choose
a subset $\cU$ of $\cX$ of cardinality $< \vk$
as in the proof of Proposition \ref{SmNormCls}.
Take an $x \in \cX \setminus \cU,$ and partition
$\cX \setminus (x \cup \cU)$ into two moieties:
$$
\cX \setminus (x \cup \cU) = \cY_0 \sqcup \cY_1.
$$
Then by Proposition \ref{MStab-in-a-small-index-subgroup} (ii) the following automorphisms
$\rho_1,\rho_2,\rho_3$ that act identically on $\cU$ belong to the centralizer
$C(\s)$ of $\s$:
\begin{alignat*} 7
\rho_1: \quad & x \to x,           && \rho_2: \quad && x \to x,               && \rho_3: \quad && x \to x\inv,\quad &&          \\
	 & y \to x\inv y x, \quad  &&               && y \to y,               &&               && y \to y,          && (y \in \cY_1),\\
	 & y \to y,                &&               && y \to x\inv y x, \quad &&               && y \to y,          && (y \in \cY_2).
\end{alignat*}
The product $\rho=\rho_3 \rho_2 \rho_1$ is an automorphism
of $F$ which fixes $\cU$ pointwise, inverts $x$
and takes every $y \in\cY_1 \cup \cY_2$ to its
conjugate by $x.$ It is proved in \cite[Lemma 4.2]{To:Towers}
that any automorphism of $F$ commuting with
$\rho$ takes $x$ either to $v x v\inv,$ or
$v x^{-1} v\inv,$ where $v$ is in the fixed-point
subgroup of $\rho,$ that is, in $\str{\cU}.$
It follows that
$$
\sigma(z) = w_0(z;v)=v z v\inv \qquad (z \in \cX),
$$
or
$$
\sigma(z) =w_1(z;v)= v z\inv v\inv \qquad (z \in \cX).
$$
But in the second case, it is not true for
the term $w_1$ that
$$
w_1(xy; v) = v (xy)\inv v\inv = w_1(x;v) \cdot w_1(y;v) =
v x\inv y\inv v\inv
$$
for every $x,y \in \cX.$ Hence $\s$ is an inner
automorphism of $F,$ as claimed.
\end{proof}

\begin{Rem}  \label{Burnside}
\em A theorem by Burnside \cite{Burnside} states that
given a centerless group $G$ such that the group $\inn G$ is a characteristic
subgroup of $\aut G,$ we have that $\aut G$
is complete. It then follows from Corollary \theCor\
that the automorphism group of any infinitely
generated free group is complete
(a result proven in \cite{To:Towers} by a different
method).
\end{Rem}

\section{Relatively free groups $F/R'$}

Recall that a {\it derivation} of a given
group $G$ in a $G$-module (a module over
the group ring $\Z[G]$) $M$ is any map
$D : G \to M$ such that
$$
D(ab) =D(a)+a D(b)
$$
for every $a,b \in G$ (here $a D(b)$
is the result of the action of a scalar $a \in G \sle \Z[G]$
on a vector $D(b) \in M.$)

As it has been proved by Fox \cite{Fox} if $F$ is a free
group with a basis $(X_i : i \in I)$ then
for any prescribed elements $Y_i \in \Z[F]$
there is a unique derivation $D$ of $F$ in $\Z[F]$ such that
$$
D(X_i)=Y_i \qquad (i \in I).
$$
In particular, for every $i \in I$ there is
a derivation $D_i$ of $F$ such that
$$
D_i(X_j) =\delta_{ij} \qquad (i,j \in I).
$$

Now let $R$ be a normal subgroup of $F$
and let $R'$ denote the commutator
subgroup of $R;$ the quotient group $R/R'$
will be denoted by $\avR.$

We shall write $\av{\phantom a}$
for the homomorphism $\Z[F] \to \Z[F/R]$
of group rings induced by the natural group homomorphism
$F \to F/R;$ it is convenient to use
the same symbol $\av{\phantom a}$\ to
denote the homomorphism
$\Z[F/R'] \to \Z[F/R]$ induced
by the natural homomorphism $F/R' \to F/R.$

Clearly, any $F/R$-module can be in a
natural way viewed as an $F$-
and as an $F/R'$-module. Consider a free $F/R$-module
$M$ with free generators $(t_i : i \in I).$
Then it is easy to see that the map
\begin{equation} \label{StDer0G}
\ptl(aR') = \sum \avst{D_i(a)} t_i
\end{equation}
where $a$ runs over $F$ is a well-defined derivation
of $F/R'$ in $M,$ since $\avst{D_i(b)}=0$
for every $b \in R'$ and
for every $i \in I.$

A famous result by Magnus from \cite{Magnus}
is tantamount to the fact that $\ptl : F/R' \to M$ is injective
(see, for instance, \cite{RemSok}). Moreover, the
following properties
\begin{align} \label{Diff0avR}
& \ptl (r_1 r_2) =\ptl(r_1)+ \av r_1 \ptl(r_2)=\ptl(r_1)+\ptl(r_2),\\
& \ptl( g\avR* r ) = \ptl( grg\inv) = \ptl(g) + \av g \ptl(r) - \avst{grg\inv} \ptl(g)
=\av g \ptl(r). \nonumber
\end{align}
are true for all $r_1,r_2 \in \avR=R/R'$ and for all
$g \in F/R'.$ One can therefore state that
$R/R'$ and $\ptl(R/R'),$ viewed as $F/R$-modules,
are isomorphic via $\ptl.$

According to a result by Auslander and Lyndon \cite{AuLyn},
if the quotient group $F/R$ is infinite, the
group $F/R'$ is centerless; we shall use this fact in
Corollary \ref{Dyer!} and Theorem \ref{MR} below.

\begin{Lem} \label{Actn_on_R/Rd}
Let $F$ be an infinitely generated free group, $R$ a fully
invariant subgroup of $F$ such that the group ring
$\Z[F/R]$ has no zero divisors and
all its units are trivial:
$$
U(\Z[F/R]) = \pm F/R.
$$
Suppose that $\s \in \aut{F/R'}$ has
small conjugacy class in $\aut{F/R'}.$ Then the restriction
of $\s$ on the group $\avR=R/R'$ coincides
with the restriction on $\avR$ of
a suitable inner automorphism of $F/R',$ that is,
there is a $v \in F/R'$ such that
$$
\s r=v r v\inv.
$$
for every $r$ in $R/R'.$
\end{Lem}

\begin{proof}
Fix a basis $\cB$ of the free group $F$ and let
$\cX$ be the image of $\cB$ under the
natural homomorphism $F \to F/R'.$
By Proposition \ref{SmNormCls}, there are
elements $\stuple u$ of $\cX$ and a term $w(*;*_1,\ldots,*_s)$
of the language of group theory such that
$w$ satisfies \eqref{w_sets_the_image} and \eqref{w_is_hom}
and
$$
\s(z)=w(z;\vec u)
$$
for all $z \in F/R'.$ Suppose that
$$
w(x; \vec u)=v_1 x^{k_1} \ldots v_m x^{k_m}
$$
where elements $v_2,\ldots,v_m$ from
the subgroup $\str{\vec u}$ generated
by the elements $\vec u = u_1,\ldots,u_s$
are nontrivial, $k_1,\ldots,k_{m-1}$
are nonzero integers, while $v_1 \in \str{\vec u}$
and $x^{k_m}$ could be equal to identity.

We show that the sum $l=k_1+\ldots+k_m$
of exponents of $x$ is $1.$ Indeed, by \eqref{w_is_hom}
\begin{equation}
w(xy; \vec u)=w(x;\vec u) w(y;\vec u),
\end{equation}
for all $x,y \in \cX.$ Assume that $x,y$
are distinct members of $\cX.$ Take an endomorphism
of $F/R'$ sending all $u_i$ to $1,$ while
preserving $x$ and $y,$ and apply it to the
both parts of (\theequation):
$$
(xy)^l = x^l y^l.
$$
Let $X$ be the element of $\cB$ whose image is
$x.$ Consider the derivation $D_X$ of $F$ which
takes $X$ to $1$ and takes to $0$ all other
elements of $\cB.$ Let then $\ptl_x$ be the derivation
of $F/R'$ in $\Z[F/R]$ induced by $D_X$:
$$
\ptl_x(aR') = \avst{D_X(a)}\qquad (a \in F).
$$
We have that
$$
\ptl_x( (xy)^l)=\ptl_x(x^l)
$$
Let, for instance, $l > 0.$ Then
\begin{align*}
\ptl_x( (xy)^l ) &= 1+ \av{xy}+\ldots+\av{xy}^{l-1} =\ptl_x(x^l) \\
                 &= 1+ \av{x}+\ldots+\av{x}^{l-1}.
\end{align*}
We apply an endomorphism of the group ring $\Z[F/R]$
induced by the endomorphism of $F/R$ fixing
all elements of $\avst{\cX} \setminus \{\av y\}$ and taking $\av{xy} \to 1$
to the both parts of the last equation:
$$
l = \ptl_x(x^l).
$$
The same is true when $l < 0.$ Thus
$$
\ptl_x(x^l)=l
$$
which means that $l=0,$ or $l=1.$ The
former is clearly impossible, since
$\s$ is an automorphism of $F/R'.$ Hence
$l=k_1+\ldots+k_m=1,$ as claimed.

Observe also that after applying to the both parts of
(\theequation) an endomorphism of $F/R$
taking both $x,y$ to $1$ and fixing
all $u_i,$ we get that
$$
w(1;\vec u)=1.
$$
In particular,
$$
v_1 v_2 \ldots v_m=1,
$$
and then
$$
w(x;\vec u) = \prod_{i=1}^m c_i x^{k_i} c_i\inv,
$$
where
$$
c_i =v_1 \ldots v_i \qquad (i=1,\ldots,m).
$$

Let $\ptl : F/R' \to \Z[F/R]$ be a derivation \eqref{StDer0G}
of $F/R'$ associated with the basis $\cB$ of $F$ we have chosen
above. By \eqref{Diff0avR}, for every $r \in R/R'$ we have that
\begin{align*}
\ptl( \s(r) ) &= \ptl( \prod_{i=1}^m c_i r^{k_i} c_i\inv)
              =\sum_{i=1}^m k_i \av c_i\, \ptl(r) \\
	      &=(\sum k_i \av c_i) \ptl(r).
\end{align*}
Let us denote the element $\sum k_i \av c_i \in \Z[F/R]$
by $f_\s.$ As the conjugacy class of the
inverse $\s\inv$ of $\s$ is also small,
the same argument applies to $\s\inv$: there
is an element $f_{\s\inv}$ of $\Z[F/R]$ with
$$
\ptl( \s\inv(r)) = f_{\s\inv} \ptl(r) \qquad [r \in R/R'].
$$
The proof of Corollary 1 in \cite{Dyer} demonstrates
that $R/R'$ is fully invariant subgroup of $F/R',$
provided that the group ring $\Z[F/R]$
has no zero divisors. Then $\s\inv(r) \in R/R'$
for every $r \in R/R'$ and hence
$$
\ptl(r) =\ptl( \s(\s\inv(r) ) )=
f_\s \ptl(\s\inv(r))=f_\s f_{\s\inv} \ptl(r),
$$
or
$$
(1-f_\s f_{\s\inv})\ptl(r)=0.
$$
As $\Z[F/R]$ has no divisors of zero,
$$
1 =f_\s f_{\s\inv}=f_{\s\inv} f_\s
$$
and as $\Z[F/R]$ has only trivial
units,
$$
f_\s = \av v, \text{ or } f_\s=-\av v
$$
for some $v \in F/R',$ whence
$$
\ptl(\s(r) ) = \av v\, \ptl(r), \text{ or }
\ptl(\s(r) ) = -\av v\, \ptl(r)
$$
for all $r \in R/R'.$ In the
first case we are done:
$$
\ptl(\s(r) ) =\ptl( v rv\inv)
$$
and $\s(r)=vrv\inv,$ since $\ptl$ is injective.
In the second case
$$
k_1 \av c_1+\ldots+k_m \av c_m = -\av v,
$$
which is impossible, since the vector
in the left-hand side has augmentation
$k_1+\ldots+k_m=1,$ whereas the vector
in the right-hand side has augmentation
$-1.$
\end{proof}

Generalizing an earlier
result by Shmel'kin \cite{Shmel'kin2} on
free solvable groups, Dyer \cite{Dyer} proved the following
result: if $F$ is a free group and a normal
subgroup $R$ is such that the quotient
group $F/R$ is torsion-free and either
is solvable, or has nontrivial center
and is not cyclic-by-periodic, then
any automorphism of the group $F/R'$
which fixes $R/R'$ pointwise is
an inner automorphism of $F/R'$ determined
by an element of $R/R'.$ We have therefore
the following corollary of Lemma \theLem.

\begin{Cor} \label{Dyer!}
Let $F$ be an infinitely generated
free group and $R$ a fully invariant
subgroup of $F$ such that the quotient
group $F/R$ satisfies the conditions
of Dyer's theorem and all units
of the group ring $\Z[F/R]$ are trivial.
Then the group $\aut{F/R'}$ is complete.
In particular, the automorphism group
of any infinitely generated free solvable
group of derived length $\ge 2$ is complete.
\end{Cor}

\begin{proof}
First, observe that as $F/R$ must be
torsion-free by the conditions, triviality
of units of $\Z[F/R]$ implies that $\Z[F/R]$
has no zero divisors (see, for instance,
\cite[\S\ 6]{Lam}.)

By Proposition \ref{S_is_def} (iv), the subgroup
$S$ of elements of $\aut{F/R'}$
whose conjugacy class is small is a characteristic
subgroup of $\aut{F/R'}.$ By Lemma \ref{Actn_on_R/Rd}
and by the quoted result by Dyer from \cite{Dyer},
$S$ equals $\inn{F/R'}.$ For any automorphism $\s$ of $F/R'$
whose conjugacy class is small is inner: the restriction of $\s$
on $R/R'$ coincides with the restriction on $R/R'$ of
a suitable inner automorphism $\tau_v$ of $F/R';$
then $\tau_{v\inv} \s$ fixes $R/R'$ pointwise,
and $\tau_{v\inv} \s =\tau_r$ for some $r \in R/R'.$
Hence $\inn{F/R'}=S$ is a characteristic
subgroup of $\aut{F/R'},$ and then the group $\aut{F/R'}$
is complete (by Burnside's theorem quoted
in Remark \ref{Burnside}).

Recall that free polynilpotent
groups (in particular, free solvable groups)
are orderable \cite{Shmel'kin}. Also,
the group ring $\Z[G]$ of an orderable
group $G$ has only trivial units \cite[\S\ 6]{Lam}.
Thus the conditions of the corollary
are met by any infinitely generated
free solvalbe group $F/F^{(k)}$ of derived length $k \ge 2,$
and hence the automorphism group $\aut {F/F^{(k)}}$
of $F/F^{(k)}$ is complete.
\end{proof}

\section{Residually torsion-free nilpotent relatively
free groups $F/R'$}

Till the end of this section $F$ will denote an infinitely
generated free group, $R$ a fully invariant subgroup of $R,$
and $G$ the quotient group $F/R'.$ We shall assume throughout
the section that the quotient group $F/R$ is residually
torsion-free nilpotent. 

Recall that if $\cP$ is a property of groups, a group $H$
is said to be {\it residually $\cP$}, if for every
nonidentity element $h$ of $H,$ there is a surjective
homomorphism from $H$ onto a group with $\cP$
such that the image of $h$ under this homomorphism
is not trivial.

By a quite standard argument, every
residually orderable group is orderable. As
any torsion-free nilpotent group is orderable,
we obtain that the group $F/R$ is orderable,
and according to the remarks we have made
at the end of the previous section the
group ring $\Z[F/R]$ is a domain
whose units are trivial. So Lemma
\ref{Actn_on_R/Rd} applies to $G=F/R'.$

As usual $\gamma_k(G)$ where $k \in \N$ denotes the $k$-th term
of the lower central series of $G$ ($\gamma_1(G)=G$
and $\gamma_{k+1}(G)=[G,\gamma_k(G)]$ for all natural
numbers $k \ge 1).$ As in \cite{DFo} we define
the series $(\av\gamma_k(G): k \ge 1)$ where
$$
\av\gamma_k(G) = \{g \in G : g^m \in \gamma_k(G) \text{ for some integer $m \ne 0$}\}.
$$
Clearly, $G/\av\gamma_k(G)$ is a torsion-free nilpotent
group of class at most $k-1.$ According to a result
by Hartley \cite[Theorem D2]{Hartley}, if $F/R$
is residually torsion-free nilpotent, so is
$F/R'.$ Thus $G$ is residually torsion-free nilpotent,
and hence
$$
\bigcap_{k \ge 1} \av \gamma_k(G) = \{1\}.
$$

For formality's sake, we shall say that a relation $X$
on a group $H$ is {\it definable} in $H,$ if
$X$ admits a description in $H$ in terms of group operation.
For instance, $X$ is definable in $H$ if $X$ is the
set of realizations in $H$ of a suitable
formula of some logic. Any definable relation on a given
group $H$ is invariant under all automorphisms of
$H.$

Working with a subgroup $H$ of $G,$ we shall denote by $I_H$ the group
$\{\tau_h : h \in H\}$ of all inner automorphisms of $G$ determined by members of $H.$

\begin{Lem} \label{SomeDefReslts}
Let $S$ be the subgroup of all automorphisms
of $G$ whose conjugacy class is small. Then

{\rm (i)} $S = \inn G \cdot S_{(\avR)}$
where $S_{(\avR)}$
is the subgroup of all elements of $S$ fixing $\avR$ pointwise;

{\rm (ii)} the subgroup $S_{(\avR)}$
is the Hirsch-Plotkin radical {\rm(}the
maximal locally nilpotent subgroup{\rm)} of the group $S;$

{\rm (iii)} if $R \le F',$ then $S\,'$ coincides
with the subgroup $I_{G'}$ of all inner
automorphisms of $G$ determined by elements
of $G'.$ In particular,
$$
I_{\avR} = S_{(\avR)} \cap S\,'.
$$

{\rm (iv)} elements of the form $\tau_x \gamma$
where $x$ is an element of $G$ whose image under the natural
homomorphism $G=F/R' \to F/R$ is a primitive
element of the group $F/R$ and $\gamma \in S_{(\avR)}$
form a definable family of the group $\aut G.$
\end{Lem}

\begin{proof}
(i) Observe that by Lemma \ref{Actn_on_R/Rd},
every element $\s$ of $S$ can be written in the
form
$$
\tau_v (\tau_v\inv \s) =\tau_v \gamma
$$
where the automorphism $\gamma=\tau_v\inv \s$ fixes $\avR$
pointwise and its conjugacy class is small.

(ii) We base our argument on the fact that $\avR=R/R'$
is the Hirsch-Plotkin radical of the group
$G=F/R'$ \cite{DFo}.

According to Corollary 2 in \cite{Dyer},
if an automorphism $\pi$ belongs
to the subgroup $\aut G_{(\avR)},$
that is, if it fixes $\avR$ pointwise,
then
$$
x\inv \pi x \in \avR
$$
for all $x \in G.$ It follows that
\begin{equation} \label{Dyer'sEq}
\pi x = x r_x
\end{equation}
where $r_x \in \avR$ for all $x \in G,$ and hence
the group $\aut G_{(\avR)}$ is abelian. As
$\avR$ is a characteristic subgroup of $G,$
the group $\aut G_{(\avR)}$ is a normal
subgroup of $\aut G.$ It follows that
$S_{(\avR)}$ is a normal abelian subgroup
of $\aut G.$

Let $\tau_g \gamma$ where $\gamma \in S_{(\avR)}$
be an element of $S$ which is not in $S_{(\avR)}.$
In particular, $g \in G\setminus \avR.$ Let $r$ be a
nonidentity element of $\avR;$ clearly,
$\tau_r \in S_{(\avR)}.$ As it is shown
in \cite{DFo},
\begin{equation} \label{NoChance}
[g,g,\ldots,g,r] \ne 1
\end{equation}
(see the proof of Theorem 3.5 in \cite{DFo}). But then
$$
[\tau_g \gamma,\tau_g \gamma ,\ldots,\tau_g \gamma,\tau_r] =
\tau_{[g,g,\ldots,g,r]} \ne \id
$$
in $\aut G.$ Hence there is no locally nilpotent subgroup
of $S$ properly containing $S_{(\avR)}.$

(ii) Recall that $I_{G'}$ denotes the group
of inner automorphisms of $G$ determined by
elements of $G'.$ Clearly, $I_{G'} \le S\,',$
since $\inn G \le S.$ Consider a commutator
of elements of $S$:
$$
\rho=\tau_a \gamma \tau_b \delta \gamma\inv \tau_a\inv \delta\inv \tau_b\inv
$$
where $\gamma,\delta \in S_{(\avR)}.$ Then
$$
\rho = \tau_a \tau_{\gamma(b)} \tau_{\delta(a\inv)} \tau_{b\inv}
$$
By \eqref{Dyer'sEq} there exist $r_b,s_a \in \avR$ such that
$\gamma(b)=b r_b$ and $\delta(a)=a s_a.$ Then $\rho$
is the inner automorphism determined by the element
$$
a b r_b s_a\inv a\inv b\inv
$$
of $G' \avR=G',$ since $R \le F'.$

(iv) By (i), $S=\inn G \cdot S_{(\avR)}$ and then
\begin{align*}
S/S_{(\avR)} &= \inn G \cdot S_{(\avR)}/S_{(\avR)}
              \cong \inn G/(\inn G \cap S_{(\avR)}) \\
              &\cong I_G/I_{\avR} \cong G/\avR \cong F/R.
\end{align*}
Thus $S/S_{(\avR)}$ is a relatively free group
isomorphic to the group $F/R.$ As the family
of all primitive elements a given relatively
free group $H$ is definable in $H,$
the result follows.
To explain in terms of group operation
that an element $z$ of $H$ is primitive, one explains
that $z$ can be included into some basis
of $H;$ a basis $X$ of $H$ being a subset of
$H$ such that any map from $X$ into $H$
can be extended to a homomorphism from
$H$ into $H$.
\end{proof}

\begin{Prop} \label{PrimConjs:Lifting}
Let $R \le F'.$ Suppose that the following conditions
are true for an automorphism $\s \in \aut G$:

{\rm (a)} the conjugacy class of $\s$ is small;

{\rm (b)} the image of $\s$ under
the natural homomorphism $S \to S/S_{(\avR)}$
is a primitive element of the relatively
free group $S/S_{(\avR)} \cong F/R;$

{\rm (c)} the group $L(\s)=\nc(\s) I_{G'}$ contains
no element of $S_{(\avR)} \setminus I_{\avR}.$

\noindent It follows that $\s$ is an inner automorphism of $G,$
that $\nc(\s) I_{G'}=\inn G,$ and that $\inn G$ is a characteristic
subgroup of $\aut G.$
\end{Prop}

\begin{proof}
Let
\begin{equation}
\s = \tau_x \gamma
\end{equation}
where $\gamma \in S_{(\avR)}.$
Suppose, towards a contradiction, that $\s$ is not
an inner automorphism of $G.$ This implies that
$\gamma \in S_{(\avR)} \setminus I_{\avR}.$
As under the natural homomorphisms $F/R' \to F/R$
and $F/R \to F/F'$ the element $x$ goes to a primitive element
of $F/F',$ there is a $c \in G'$ such
that $c x$ is a primitive element of $G.$

The group $L(\s)$ contains then the element
$$
\tau_c \tau_x \gamma = \tau_{cx} \gamma.
$$
This enables us to assume without loss of generality
that $x$ in (\theequation) is already a primitive element of $G.$

Since (c) is satisfied by $\s,$ whenever
elements of the form $\tau_a \gamma_1$ and
$\tau_a \gamma_2$ where $a \in G$
and $\gamma_1,\gamma_2 \in S_{(\avR)}$ both
belong to $L(\s),$ the elements $\gamma_1,\gamma_2$
must be congruent modulo $I_{\avR}$:
\begin{equation} \label{CongModIR}
\tau_a\gamma_1, \tau_a\gamma_2 \in L(\s) \To \gamma_1 \equiv \gamma_2 \Mod{I_{\avR}};
\end{equation}
otherwise $\gamma_2\inv \tau_a\inv \cdot \tau_a \gamma_1 \in
S_{(\avR)} \setminus I_{\avR},$ contradicting (c).

In particular, for any automorphism $\pi$ of $G$ stabilizing
our primitive element $x,$ we have by (\theequation) that
$$
\gamma^\pi \equiv \gamma \Mod{I_{\avR}}.
$$
where $\gamma^\pi = \pi \gamma \pi\inv.$

Fix a basis $\cX$ of $G$ containing $x.$ By Proposition
\ref{SmNormCls} and by \eqref{Dyer'sEq}
$$
\gamma(t) =t v(t; u_1,\ldots,u_k) \qquad (t \in \cX)
$$
where $\vec u = u_1,\ldots,u_k$ are some (fixed)
members of $\cX,$ $v(*;*_1,\ldots,*_k)$ is a term/word of the language
of group theory such
$$
zt v(zt;\vec u) = zv(z;\vec u) \cdot tv(t;\vec u)
$$
for all $z,t\in \cX$ and
$$
v(t;\vec u) \in \avR
$$
for all $t \in \cX.$

1) {\it Suppose first that $x$ does not belong to
the tuple $\vec u.$} Consider then an automorphism
$\pi$ of $G$ acting on $\cX$ as a permutation
which fixes $x$ and takes $\vec u$ to a tuple
$\pi \vec u$ having no common element with $\vec u$:
$
\pi \vec u \cap \vec u =\varnothing.
$
By \eqref{CongModIR}, there is an $s \in \avR$ such that
$
\gamma^\pi = \tau_s \gamma.
$
This implies that
$$
z v(z; \pi \vec u) = szs\inv v(z;\vec u)
$$
for all $z \in \cX.$ Since $\cX$ is infinite,
there is a $t \in \cX$ such that the letter
$t$ does not appear in the word $s,$ nor
$t \in \vec u,$ nor $t \in \pi \vec u.$ We then
apply an endomorphism of $G$ taking all elements
$\pi \vec u$ to $1$ and fixing all other elements
of $\cX$ to the both parts of the equation
$$
t v(t; \pi \vec u) = sts\inv v(t;\vec u),
$$
thereby getting that
$$
t = s_0 t s_0\inv v(t;\vec u),
$$
or
$$
s_0\inv t s_0 = t v(t; \vec u).
$$
It follows that
$$
s_0\inv z s_0 = z v(z; \vec u).
$$
for every $z \in \cX,$ and then $\gamma$
is an inner automorphism determined
by an element $s_0\inv \in \avR,$ a contradiction.

2) {\it Suppose now that $x$ is a member of $\vec u$} and $\vec u = x, u_2, \ldots, u_k.$
Write $\vec u_0$ for the tuple $u_2,\ldots,u_k.$
As above, we consider an automorphism $\pi$ of $G$
acting on $\cX$ as a permutation, fixing $x$
and such that tuples $\vec u_0$ and $\pi \vec u_0$
are disjoint. Then
$$
z v(z; x, \pi \vec u_0) = szs\inv v(z;x,\vec u_0)
$$
for all $z \in \cX.$ Working
with the endomorphism of $G$ taking all elements
$\pi \vec u_0$ to $1$ and fixing pointwise
$\cX \setminus \{\pi\vec u_0\},$ we see that
$$
t w(t; x) = s_0 t s_0\inv v(t; x, \vec u_0),
$$
or
\begin{equation}
s_0\inv t s_0 w(t; x) = t v(t; x, \vec u_0),
\end{equation}
where $w(*;*_1)$ is a fixed group word/term of the language
of groups, for all $t \in \cX$ (at first for $t \in \cX$ that are
not members of $\vec u_0,$ $\pi u_0,$ and the
set of letters of $\cX$ forming $s,$ then
for all $t \in \cX$.)

The equation (\theequation) means that
$$
\gamma =\tau_{s\inv} \delta
$$
where
$$
\delta t = t w(t;x)\qquad (t \in \cX).
$$
Clearly, $\delta$ is in $S_{(\avR)},$ is not
an inner automorphism of $G,$ and the element
$\tau_x \delta$ is a member of $L(\s).$

Our goal is to show
that $\delta$ is the identity automorphism;
this will imply, as in 1) above, that $\gamma$ is an inner
automorphism, which is impossible.

\begin{claim-num} \label{SecComp}
For every $t \in \cX$ and for every natural
number $k$
$$
w(t;x^k)=w(t;x)^k.
$$
\end{claim-num}

Let $y$ be a member of $\cX$ which is not
equal to $x.$ We start with some two elements
of $L(\s)$ of the form $\tau_y \eta$
where $\eta \in S_{(\widehat R)}$ to gain more information
about $w(*;*_1).$ First, we see that
$\tau_y \delta^\pi$ belongs to $L(\s)$
where $\pi \in \aut G$ interchanges
$x$ and $y,$ while fixing all other
elements of $\cX.$ Second, let $\rho$
be the automorphism of $G$ which
takes  $x$ to $xy$ and fixes
$\cX \setminus \{x\}$ pointwise.
Then
$$
(\tau_x \delta)^\rho \tau_x\inv \delta\inv =
\tau_{xy r\inv x\inv } \delta^\rho \delta\inv \in L(\s),
$$
where $\delta^\rho(x)=xr$ and $r \in \avR.$
As for a suitable element $c$ of $G'$ we have
that $c xy r\inv x\inv =y,$ it follows that
$\tau_y \delta^\rho \delta\inv$ is also
in $L(\s).$ Hence by \eqref{CongModIR},
there exists an $s \in \avR$ with
$$
\delta^\rho \delta\inv = \tau_s \delta^\pi.
$$
Comparing the images of a $t \in \cX \setminus \{x\}$
under the automorphisms participating in the both
parts of the last equation, one obtains that
\begin{equation}
tw(t;xy) w(t;x)\inv = sts\inv w(t;y).
\end{equation}
In particular,
$$
y w(y;xy) w(y;x)\inv = sys\inv.
$$
for $t=y.$

Let $k > 1.$ Consider the endomorphism $\eps$ of $G$
taking $y$ to $x^k$ and fixing $\cX \setminus \{y\}$
pointwise. Then
$$
x^k w(x^k;x^{k+1}) w(x^k;x)\inv = \eps(s) x^k \eps(s)\inv,
$$
whence $x^k = \eps(s) x^k \eps(s)\inv.$
Clearly, $x^k \not\in \avR,$ since $F/R$ is torsion-free
and $\eps(s) \in \avR.$ Hence $\eps(s)=1.$ We then apply $\eps$ to the
both parts of (\theequation), assuming
that $t$ is an arbitrary element of $\cX \setminus \{x\}$:
$$
t w(t;x^{k+1}) w(t;x)\inv= t w(t;x^k).
$$
By the induction hypothesis $w(t;x^k)=w(t;x)^k$
and the result follows.

\begin{claim-num} \label{WeAre_in_gamma3}
For every $t \in \cX$
the element $w(t;x)$ is in $\gamma_3(G),$ the third
term of the lower central series
of $G.$
\end{claim-num}

As we observed above,
$\delta$ fixes $\avR$ pointwise,
and adding to that the fact that
$\delta$ has small conjugacy class, we get that
$$
\delta(r) =r w(r;x)=r,
$$
for all $r \in \avR,$ whence
$w(r;x)=1.$ Also $w(1,x)=1$ and then we can write $w(t;x)$ where $t \in \cX$ as
a product of conjugates of powers of $t$:
$$
w(t;x) = \prod_{i=1}^n x^{k_i} t^{s_i} x^{-k_i}.
$$
Since $w(t,x) \in \widehat R \le G',$ the sum of exponents
$s_i$ is zero. Substitute an arbitrary $r \in \avR$ for $t$
in the last equality and take the
standard derivative, keeping in mind
that $w(r;x)=1$:
$$
0=\ptl( w(r;x) )
= \sum_{i=1}^n s_i \av x^{k_i} \ptl(r)
= \left( \sum_{i=1}^n s_i \av x^{k_i} \right) \ptl(r).
$$
Therefore
\begin{equation}
\sum_{i=1}^n s_i \av x^{k_i} =0,
\end{equation}
Suppose that there are exactly $l$ pairwise
distinct exponents $k_i$ participating
in (\theequation), say, $m_1,\ldots,m_l.$
Due to linear independence of powers
of $\av x$ over $\Z,$ there must be a partition
of $\{1,2,\ldots,n\}$ into $l$ pairwise
disjoint sets
$$
\{1,2,\ldots,n\} = A_1 \sqcup \ldots \sqcup A_l
$$
such that for a particular $j,$ for every $i_1,i_2 \in A_j,$
integers $s_{i_1},s_{i_2}$ are coefficients
of $x^{m_j}$ in (\theequation) and
$$
\sum_{i \in A_j} s_i =0.
$$

Observe that in a nilpotent group $H$ of class two all conjugates
of a given element of $H$ are commuting. Then
we have for every $t \in \cX$:
\begin{align*}
w(t;x) &= \prod_{i=1}^n x^{k_i} t^{s_i} x^{-k_i} \equiv \prod_{j=1}^l \prod_{i \in A_j} (x^{m_j} t^{s_i} x^{-m_j}) \Mod{\gamma_3(G)} \\
       &\equiv \prod_{j=1}^l (x^{m_j} t^{\sum_{i \in A_j}s_i} x^{-m_j})
	\equiv \prod_{j=1}^l (x^{m_j} t^0 x^{-m_j}) \equiv 1 \Mod{\gamma_3(G)}.
\end{align*}

\begin{claim-num}
For every $k \ge 3$ and for every $t \in \cX \setminus \{x\}$
$$
w(t;x) \equiv 1 \Mod{\av\gamma_k(G)}.
$$
Hence $w(t;x)=1$ and $\delta$ is the
identity automorphism.
\end{claim-num}

Claim \ref{WeAre_in_gamma3} takes care of the induction
base. Assume that $w(t;x) \in \av\gamma_k(G),$ that
is,
\begin{equation} \label{Sitting_in_gamma_k}
w(t;x)^m =w(t;x^m) \in \gamma_k(G)
\end{equation}
for some natural number $m > 1$
(the equality is justified by Claim \ref{SecComp}).

Observe that the subgroup $\av\gamma_k(G)$
is invariant under all endomorphisms of $G.$
As $\delta$ has small conjugacy class,
$\delta(z) = zw(z;x)$ for every $z \in G$
by Proposition \ref{SmNormCls}. Hence
for every $t \in \cX$ and for every $q \in \Z$
$$
\delta(t^q) =(t w(t;x))^q =t^q w(t^q;x)
$$
Let $\eps$ be an endomorphism of $G$ taking
$x$ to $x^{ml}$ where $l \in \Z$ and stabilizing
every element of $\cX \setminus \{x\}.$
After application of $\eps$ to the last
equality, we see that
$$
t^q w(t^q, x^{ml}) = (t w(t;x^{ml}))^q.
$$
Since elements of $\gamma_k(G)$ commute
modulo $\gamma_{k+1}(G)$ with all elements
of $G$ and since $w(t;x^{ml})=w(t;x^m)^l \in \gamma_k(G)$
by Claim \ref{SecComp},
$$
t^q w(t^q, x^{ml}) = (t w(t;x^{ml}))^q \equiv t^q w(t;x^{ml})^q \Mod{\gamma_{k+1}(G)}.
$$
Then, again by Claim \ref{SecComp},
\begin{equation} \label{Powers:1st,2nd}
w(t^q,x^{ml}) \equiv w(t;x)^{qml} \Mod{\gamma_{k+1}(G)}.
\end{equation}

Due to invariance of $\gamma_k(G)$ under
endomorphisms of $G,$ we obtain from
\eqref{Sitting_in_gamma_k} that
$w(t^m;x^m) \in \gamma_k(G).$ Therefore
$w(t^m;x^m)$ can be written as a product
of basis commutators of weight $k$ modulo
$\gamma_{k+1}(G)$:
$$
w(t^m;x^m) \equiv \prod_i b_i(t,x) \Mod{\gamma_{k+1}(G)}.
$$
Consider an endomorphism of $G$ taking
both $x$ and $t$ to their squares and apply
it to the last congruence:
\begin{equation}
w(t^{2m};x^{2m}) \equiv \prod_i b_i(t^2,x^2) \Mod{\gamma_{k+1}(G)}.
\end{equation}
It is easy to see that $b_i(t^2,x^2) \equiv b_i(t,x)^{2^k} \Mod{\gamma_{k+1}(G)};$
for instance,
$$
[t^2,x^2,t^2] \equiv [t,x^2,t^2]^2 \equiv
[t,x,t^2]^{2^2} \equiv [t,x,t]^{2^3} \Mod{\gamma_4(G)}.
$$
This implies that the element in the right-hand
side of (\theequation) is congruent to $w(t^m;x^m)^{2^k},$
and, further, to $w(t,x)^{m^2 2^k}$ by \eqref{Powers:1st,2nd}.
By the same equation \eqref{Powers:1st,2nd}, the
element in the left-hand side of (\theequation) is
congruent to $w(t;x)^{4m^2}.$ Therefore
$$
w(t;x)^{4m^2} \equiv w(t;x)^{2^k m^2} \Mod{\gamma_{k+1}(G)},
$$
or
$$
w(t;x)^{(2^k-4)m^2} \equiv 1 \Mod{\gamma_{k+1}(G)}.
$$
As $k \ge 3,$ $2^k > 4,$ and we are done.
\end{proof}

\begin{Th} \label{MR}
Let $F$ be an infinitely generated free group,
$R$ a fully invariant subgroup of $F$ which
is contained in the commutator subgroup
of $F.$ Suppose that the quotient group
$F/R$ is residually torsion-free nilpotent. Then
the automorphism group $\aut{F/R'}$
of the group $F/R'$ is complete.
\end{Th}

\begin{proof}
By Burnside's theorem and by Proposition
\ref{PrimConjs:Lifting}.
\end{proof}

{\it Acknowledgements.} The author would like to thank Edward Formanek and Vladimir
Shpilrain for helpful information.

\end{document}